\documentclass{amsart}
\usepackage[T1]{fontenc}   

\begin{document}

\title[$P_{\aleph_{1}}$-Proximities]%
{A Generalization of the Notion of a $P$-space to Proximity Spaces}

\author{Joseph Van Name}
\address{Department of Mathematics \& Statistics; University of South Florida;
Tampa, Florida, 33620}
\email{jvanname@mail.usf.edu}


\subjclass[2010]{Primary 54E05; Secondary 28A05,54G10}

\keywords{Proximity space, Baire $\sigma$-algebra, $P$-space}

\newtheorem{theorem}{Theorem}[section]
\newtheorem{lemma}[theorem]{Lemma}
\newtheorem{proposition}[theorem]{Proposition}
\newtheorem{corollary}[theorem]{Corollary}
\theoremstyle{definition}
\newtheorem{definition}[theorem]{Definition}

\newtheorem{remark}[theorem]{Remark}
\newtheorem{example}[theorem]{Example}
\numberwithin{equation}{section}

\begin{abstract}
In this note, we shall generalize the notion of a $P$-space to proximity
spaces and investigate the basic properties of these proximities.
We therefore define a $P_{\aleph_{1}}$-proximity to be a proximity where if $A_{n}\prec B$ for all $n\in\mathbb{N}$,
then $\bigcup_{n}A_{n}\prec B$. It turns out that the class of $P_{\aleph_{1}}$-proximities is equivalent to
the class of $\sigma$-algebras. Furthermore, the $P_{\aleph_{1}}$-proximity coreflection of a proximity space is the
$\sigma$-algebra of proximally Baire sets.
\end{abstract}

\maketitle

\section{\bf Introduction}
We begin by reviewing basic facts on proximity spaces without proofs. All our preliminary
information on proximity spaces can be found in \cite{N}. In this paper, we shall assume all
proximity spaces are separated and all topologies are Tychonoff.
If $\delta$ is a relation, then we shall write $\overline{\delta}$
for the negation of the relation $\delta$. In other words, we have $R\overline{\delta}S$ if and only
if we do not have $R\delta S$. The complement of a subset $A$ of a set $X$ will be denoted by
$A^{c}$.

A \textit{proximity space} is a pair $(X,\delta)$ where $X$ is a set and $\delta$ is a
relation on the power set $P(X)$ that satisfies the following axioms.

1. $A\delta B$ implies $B\delta A$.

2. $(A\cup B)\delta C$ if and only if $A\delta C$ or $B\delta C$.

3. If $A\delta B$, then $A\neq\emptyset$ and $B\neq\emptyset$.

4. If $A\overline{\delta}B$, then there is a set $E$ such that $A\overline{\delta}E$ and
$E^{c}\overline{\delta}B$.

5. If $A\cap B\neq\emptyset$, then $A\delta B$.

A proximity space is \textit{separated} if and only if $\{x\}\delta\{y\}$ implies $x=y$. 

Intuitively, $A\delta B$ whenever the set $A$ touches the set $B$ in some sense. Therefore
a proximity space is a set with a notion of whether two sets are infinitely close to each
other.

If $(X,\delta)$ is a proximity space, then let $\prec_{\delta}$ (or, simply, $\prec$) be the binary relation on
$P(X)$ where $A\prec B$ if and only if $A\overline{\delta}B^{c}$. The relation $\prec$ satisfies the following.

1. $X\prec X$.

2. If $A\prec B$, then $A\subseteq B$.

3. If $A\subseteq B,C\subseteq D,B\prec C$, then $A\prec D$.

4. If $A\prec B_{i}$ for $i=1,\ldots,n$, then $A\prec\bigcap_{i=1}^{n}B_{i}$.

5. If $A\prec B$, then $B^{c}\prec A^{c}$.

6. If $A\prec B$, then there is some $C$ with $A\prec C\prec B$.

If $\prec$ satisfies $1-6$ and if we define $\delta$ by letting
$A\delta B$ if and only if $A\prec B^{c}$, then $(X,\delta)$ is a proximity space.

If $(X,\delta)$ is a proximity space, then we put a topology $\tau_{\delta}$ on $X$ by letting
$\overline{A}=\{x|x\delta A\}$. A set $U\subseteq X$ is open if and only if $\{x\}\prec U$ whenever
$x\in U$. 

If $(X,\delta)$ and $(Y,\rho)$ are proximity spaces, then a function $f:X\rightarrow Y$ is a \textit{proximity
map} if $f(A)\rho f(B)$ whenever $A\delta B$. It can easily be shown that $f$ is a proximity map
if and only if $f^{-1}(C)\overline{\delta}f^{-1}(D)$ whenever $C\overline{\rho}D$. Furthermore, $f$ is a proximity
map if and only if $f^{-1}(C)\prec_{\delta}f^{-1}(D)$ whenever $C\prec_{\rho}D$. Every proximity map is continuous.

\begin{example}
If $X$ is a set, then $\delta$ is the discrete proximity if $A\overline{\delta}B$ whenever
$A\cap B=\emptyset$.
\end{example}
\begin{example}
Let $(X,\tau)$ be a Tychonoff space. Let $A\overline{\delta}B$ if there is a continuous function 
$f:(X,\tau)\rightarrow[0,1]$ with $f(A)\subseteq\{0\},f(B)\subseteq\{1\}$. Then $(X,\delta)$ is a proximity space
that induces the topology on $X$ (i.e., $\delta$ is \emph{compatible} with $\tau$). It is well known that a topology $X$ is induced by some
proximity if and only if $X$ is Tychonoff.
\end{example}
\begin{example}
Compact spaces have a unique compatible proximity where $A\delta B$ if and only if $\overline{A}\cap\overline{B}\neq\emptyset.$
Furthermore, if $(X,\delta)$ is a compact proximity space, and $(Y,\rho)$ is a proximity space, then
a map $f:X\rightarrow Y$ is continuous if and only if $f$ is a proximity map.
\end{example}

If $(X,\delta)$ is a proximity space and $Y\subseteq X$, then define a relation $\delta_{Y}$
on $P(Y)$ by letting $A\delta_{Y}B$ if and only if $A\delta B$. Then $\delta_{Y}$ is a proximity on $Y$
that induces the subspace topology on $Y$ called the induced proximity.

If $(X,\delta)$ is a proximity space, then $A\overline{\delta}B$ if and only if there is a proximity map
$g:(X,\delta)\rightarrow[0,1]$ with $g(A)\subseteq\{0\},g(B)\subseteq\{1\}$.

If $(X,\delta)$ is a proximity space, then there is a unique compactification $\mathcal{C}$ of $X$
where $A\delta B$ if and only if $(\textrm{cl}_{\mathcal{C}}A)\cap(\textrm{cl}_{\mathcal{C}}B)\neq\emptyset$. This compactification is called the \textit{Smirnov compactification} of $X$ and the proximity $\delta$ is the proximity induced by the unique
proximity on the compact space $\mathcal{C}$. If $X$ is a Tychonoff space,
then the proximity spaces that induce the topology on $X$ are in a one-to-one correspondence with
the compactifications of $X$. If $(X,\delta),(Y,\rho)$ are proximity spaces and
$\mathcal{C}$ is the Smirnov compactification of $X$ and $\mathcal{D}$ is the Smirnov compactification
of $Y$, then a function $f:X\rightarrow Y$ is a proximity map if and only if $f$ has a unique extension
to a continuous map from $\mathcal{C}$ to $\mathcal{D}$.

An algebra of sets $(X,\mathcal{M})$ is reduced if and only if whenever $x,y\in X$ are distinct
then there is an $R\in\mathcal{M}$ with $x\in R,y\in R^{c}$. We assume that all algebras
of sets are reduced. If $X$ is a topological space, then a \textit{zero set} is a set of the form
$f^{-1}(0)$ where $f:X\rightarrow\mathbb{R}$ is continuous. The union of finitely many zero sets
is a zero set, and the intersection of countably many zero sets is a zero set.
A \textit{$P$-space} is a Tychonoff space where every $G_{\delta}$-set is open. It is
well known and it can easily be shown that a Tychonoff space is a $P$-space if and only if
every zero set is open. If $X$ is a Tychonoff space, then the \textit{$P$-space coreflection} $(X)_{\aleph_{1}}$ is the space with underlying set $X$ and where the $G_{\delta}$-sets
in $X$ form a basis for the topology on $(X)_{\aleph_{1}}$.

\section{\bf $P_{\aleph_{1}}$-proximities}

A separated proximity space $(X,\delta)$ is a \textit{$P_{\aleph_{1}}$-proximity space}
if whenever $A_{n}\subseteq X$ for $n\in\mathbb{N}$ and $B\subseteq X$
and $\bigcup_{n=0}^{\infty}A_{n}\delta B$, then $A_{n}\delta B$ for
some $n$. In other words, $X$ is a $P_{\aleph_{1}}$-proximity space if and only if whenever
$A_{n}\prec B$ for each natural number $n$, then $\bigcup_{n}A_{n}\prec B$. Equivalently,
$X$ is a $P_{\aleph_{1}}$-proximity if and only if whenever $A\prec B_{n}$ for all $n$, then
$A\prec\bigcap_{n}B_{n}$. 

A proximity space $(X,\delta)$ is said to be \textit{zero-dimensional}
if and only if whenever $A\overline{\delta}B$, then there is a $C\subseteq X$ with
$A\overline{\delta}C,B\overline{\delta}C^{c},C\overline{\delta}C^{c}$. In other words,
$(X,\delta)$ is zero-dimensional if and only if whenever $A\prec B$ there is a $C$ with
$A\prec C\prec C\prec B$. If $(X,\delta)$ is a proximity space, then let $\mathcal{M}_{\delta}
=\{R\subseteq X|R\overline{\delta}R^{c}\}=\{R\subseteq X|R\prec R\}$. If $(X,\mathcal{M})$ is an algebra
of sets, then let $\delta_{\mathcal{M}}$ be the relation on $P(X)$ where
$U\overline{\delta_{\mathcal{M}}}V$ if and only if there is some $R\in\mathcal{M}$ with
$U\subseteq R,V\subseteq R^{c}$. 

\begin{theorem}
\begin{enumerate}
\item If $(X,\delta)$ is a proximity space, then $(X,\mathcal{M}_{\delta})$ is an
algebra of sets.
\item If $(X,\mathcal{M})$ is an algebra of sets, then $(X,\delta_{\mathcal{M}})$ is a zero-dimensional
proximity space.
\item If $\delta$ is a zero-dimensional proximity, then $\delta_{\mathcal{M}_{\delta}}=\delta$.
\item If $(X,\mathcal{M})$ is an algebra of sets, then
$\mathcal{M}=\mathcal{M}_{\delta_{\mathcal{M}}}$.
\item If $(X,\delta)$ is a zero-dimensional proximity space, then $\mathcal{M}_{\delta}$
is a basis for the topology on $X$.

\end{enumerate}
\label{BezDual}
\end{theorem}
\begin{proof}
See \cite{B}.
\end{proof}
\begin{theorem}
Let $(X,\mathcal{M}),(Y,\mathcal{N})$ be algebras of sets. Then a mapping
$f:(X,\delta_{\mathcal{M}})\rightarrow(Y,\delta_{\mathcal{N}})$ is a proximity map if and only if $f^{-1}(U)\in\mathcal{M}$ for each 
$U\in\mathcal{N}$.

\label{ProxMeasTran}
\end{theorem}
\begin{proof}
$\rightarrow$ Assume that $f$ is a proximity map. For each
$R\in\mathcal{N}$ we have $R\prec R$, so $f^{-1}(R)\prec f^{-1}(R)$, and
thus $f^{-1}(R)\in\mathcal{M}$.

$\leftarrow$ Let $U,V\subseteq Y$, and let $U\prec V$. Then there is an
$R\in\mathcal{N}$ with $U\subseteq R\subseteq V$. Therefore
$f^{-1}(R)\in\mathcal{M}$, and $f^{-1}(U)\subseteq f^{-1}(R)\subseteq f^{-1}(V)$.
Therefore $f^{-1}(U)\prec f^{-1}(V)$, so $f$ is a proximity map.
\end{proof}

\begin{theorem}
Every $P_{\aleph_{1}}$-proximity space is a zero-dimensional proximity space.
\end{theorem}
\begin{proof}
Let $X$ be a $P_{\aleph_{1}}$-proximity space. Assume $A\prec B$. Then
there is a sequence $(C_{n})_{n\in\mathbb{N}}$ of subsets of $X$ where
$A=C_{0}\prec C_{1}\prec...\prec C_{n}\prec...\prec B$. Therefore let
$C=\bigcup_{n\in\mathbb{N}}C_{n}$. Since $X$ is a $P_{\aleph_{1}}$-proximity space,
we have $A\prec C\prec B$. Furthermore, since $C_{n}\prec C$ for all $n$,
and $X$ is a $P_{\aleph_{1}}$-proximity space, we have $C=\bigcup_{n}C_{n}\prec C$.
Therefore $X$ is a zero-dimensional proximity space.
\end{proof}
\begin{theorem}
An algebra of sets $(X,\mathcal{M})$ is a $\sigma$-algebra if and only if
$(X,\delta_{\mathcal{M}})$ is a $P_{\aleph_{1}}$-proximity space.

\label{SigPProx}
\end{theorem}
\begin{proof}
$\rightarrow$ Assume $(X,\mathcal{M})$ is a $\sigma$-algebra.
Assume that $A_{n}\prec B$ for all $n$. Then for each $n$ there is an
$C_{n}\in\mathcal{M}$ with $A_{n}\subseteq C_{n}\subseteq B$. Therefore
$\bigcup_{n}A_{n}\subseteq\bigcup_{n}C_{n}\subseteq B$, so $\bigcup_{n}A_{n}\prec B$.

$\leftarrow$ Assume that $\delta_{\mathcal{M}}$ is a $P_{\aleph_{1}}$-proximity space,
and let $R_{n}\in\mathcal{M}$ for each natural number $n$. Then $R_{n}\prec R_{n}\subseteq\bigcup_{n}R_{n}$,
so $\bigcup_{n}R_{n}\prec\bigcup_{n}R_{n}$, so $\bigcup_{n}R_{n}\in\mathcal{M}$.
Therefore $(X,\mathcal{M})$ is a $\sigma$-algebra.
\end{proof}
\begin{corollary}
\begin{enumerate}

\item The category of zero-dimensional proximity spaces with proximity maps is isomorphic to the category of
reduced algebras of sets with measurable maps (i.e. maps $f:(X,\mathcal{N})\rightarrow(Y,\mathcal{M})$ where if
$R\in\mathcal{M}$, then $f^{-1}(R)\in\mathcal{N}$).

\item Their corresponding subcategories of $P_{\aleph_{1}}$-proximity spaces and reduced $\sigma$-algebras are isomorphic as well.
\end{enumerate}
\end{corollary}

Now given a proximity space $(X,\delta)$, we shall characterize the smallest $\sigma$-algebra $(X,\mathcal{M})$
such that the identity function from $(X,\delta_{\mathcal{M}})$ to $(X,\delta)$ is a proximity map,
but we must first generalize the notion of a zero set
to proximity spaces. Let $(X,\delta)$ be a proximity space. Then a \textit{proximally zero set} is
a set of the form $f^{-1}(0)$ for some proximity map $f:(X,\delta)\rightarrow[0,1]$.
If $\mathcal{C}$ is the Smirnov compactification of $X$, then $f$ has a unique extension to
a continuous function $\hat{f}:\mathcal{C}\rightarrow[0,1]$. Hence 
$f^{-1}(0)=\hat{f}^{-1}(0)\cap X$. Therefore the proximally zero sets
on a proximity space are precisely the sets of the form $X\cap Z$
where $Z\subseteq\mathcal{C}$ is a zero set. As a consequence, the intersection of countably
many proximally zero sets is a proximally zero set, and the union of finitely
many proximally zero sets is a proximally zero set.

The $\sigma$-algebra $\mathcal{B}^{*}(X,\delta)$ of \textit{proximally Baire sets} on a proximity space $(X,\delta)$ is the smallest
$\sigma$-algebra containing the proximally zero sets.
If $(X,\delta)$ is a proximity space with Smirnov compactification
$\mathcal{C}$, and $\mathcal{M}$ is the Baire $\sigma$-algebra on $\mathcal{C}$, then $\{R\cap X|R\in\mathcal{M}\}$
is the $\sigma$-algebra of proximally Baire sets on $X$.

\begin{remark}
Every proximally zero set is a zero set, but in general there are zero sets that are not
proximally zero sets. For example, let
$A$ be an uncountable discrete space, and let $\delta$ be the proximity induced by
the one point compactification $A\cup\{\infty\}$. It is well known that for normal spaces
the closed $G_{\delta}$-sets are precisely the zero sets, so it suffices to characterize the
closed $G_{\delta}$-subsets of $A\cup\{\infty\}$. Let $R\subseteq A\cup\{\infty\}$
be a closed $G_{\delta}$-set. If $\infty\not\in R$, then $R$ is finite. If $\infty\in R$,
then $R=\bigcap_{n}U_{n}$ for some sequence of open sets $U_{n}\subseteq A\cup\{\infty\}$,
but each $U_{n}$ is co-finite, so $R$ is co-countable. Therefore each zero set in $A\cup\{\infty\}$
is either co-finite or co-countable. Hence every proximally zero set in $A$ is either co-finite
or co-countable. We conclude that not every zero set in $A$ is a proximally zero set in $A$.
\end{remark}

\begin{theorem}
Let $(X,\delta)$ be a proximity space. Then a set $Z\subseteq X$ is a proximally zero
set if and only if there is a sequence $(Z_{n})_{n\in\mathbb{N}}$ with $Z=\bigcap_{n}Z_{n}$ and
where $\ldots Z_{n}\prec Z_{n-1}\prec\ldots\prec Z_{1}$.
\end{theorem}
\begin{proof}
$\rightarrow$ If $Z\subseteq X$ is a proximally zero set, then there is a proximity
map $f:(X,\delta)\rightarrow[0,1]$ with $Z=f^{-1}(0)$. For all $n\geq 1$, we have
$[0,\frac{1}{n+1}]\prec[0,\frac{1}{n}]$, so $f^{-1}([0,\frac{1}{n+1}])\prec f^{-1}([0,\frac{1}{n}])$,
and $Z=f^{-1}(\{0\})=\bigcap_{n}f^{-1}([0,\frac{1}{n}])$.

$\leftarrow$ Suppose that $(Z_{n})_{n\in\mathbb{N}}$ is such a sequence. Then
for all $n$ we have $Z_{n+1}\overline{\delta}Z_{n}^{c}$, so
there is a proximity map $f_{n}:(X,\delta)\rightarrow[0,1]$ with
$Z_{n+1}\subseteq f_{n}^{-1}(0)$ and $Z_{n}^{c}\subseteq f_{n}^{-1}(1)\subseteq f_{n}^{-1}(0)^{c}$.
Thus $Z_{n+1}\subseteq f^{-1}_{n}(0)\subseteq Z_{n}$. Therefore
$\bigcap_{n}Z_{n}=\bigcap_{n}f_{n}^{-1}(0)$ is a proximally zero set being
the intersection of countably many proximally zero sets.
\end{proof}
\begin{corollary}
A proximity space $(X,\delta)$ is a $P_{\aleph_{1}}$-proximity space if and only if $\mathcal{M}_{\delta}$
contains each proximally zero set.

\label{PProxZero}
\end{corollary}
\begin{proof}
$\rightarrow$ Let $(X,\delta)$ be a $P_{\aleph_{1}}$-proximity space.
If $Z$ is a proximally zero set, then there is a sequence $(Z_{n})_{n\in\mathbb{N}}$ with
$Z=\bigcap_{n\in\mathbb{N}}Z_{n}$ and where $Z_{n+1}\prec Z_{n}$ for all $n$. Therefore
$Z\prec Z_{n}$ for all $n$, so $Z\prec\bigcap_{n\in\mathbb{N}}Z_{n}=Z$, so $Z\in\mathcal{M}_{\delta}$.

$\leftarrow$ Assume $\mathcal{M}_{\delta}$ contains each proximally zero set.
Assume $A\prec B_{n}$ for all $n\geq 0$. For $n\geq 0$ there is a proximally zero
set $Z_{n}$ with $A\subseteq Z_{n}\subseteq B_{n}$, so $\bigcap_{n}Z_{n}$ is a proximally
zero set, so $A\subseteq\bigcap_{n}Z_{n}\prec\bigcap_{n}Z_{n}\subseteq\bigcap_{n}B_{n}$.
\end{proof}
If $X$ is a topological space, then a \textit{cozero set} is a set of the form $f^{-1}(0,1]$ for some
continuous $f:X\rightarrow[0,1]$.
If $(X,\delta)$ is a proximity space, then a \textit{proximally cozero set} is a set of the form
$f^{-1}(0,1]$ for some proximity map $f:(X,\delta)\rightarrow[0,1]$. In other words, a cozero set is a complement of
a zero set, and a proximally cozero set is a complement of a proximally zero set.
\begin{theorem}
\begin{enumerate}
\item\label{LindCoz} Every Lindelöf open subset of a proximity space is a proximally cozero set.

\item In a Lindelöf proximity space, every cozero set is a proximally cozero set.
\end{enumerate}
\end{theorem}
\begin{proof}
\begin{enumerate}
\item Let $(X,\delta)$ be a proximity space, and let $U$ be a Lindelöf open subset of $X$. Then for each
$x\in X$, there is a proximity map $f_{x}:(X,\delta)\rightarrow[0,1]$ such that $f_{x}(x)=1$ and $f_{x}(U^{c})\subseteq\{0\}$.
If $U_{x}=f_{x}^{-1}(0,1]$, then $U_{x}$ is a proximally cozero set with $x\in U_{x}\subseteq U$.
Since $\{U_{x}|x\in X\}$ covers $U$, there is a countable subcover $\{U_{x_{n}}|n\in\mathbb{N}\}$ of the set $U$.
Therefore $U=\bigcup_{n}U_{x_{n}}$ is a proximally cozero set since $U$ is the countable union
of proximally cozero sets.

\item Assume $(X,\delta)$ is a Lindelöf proximity space. If $U\subseteq X$ is a cozero set, then
$U$ is an $F_{\sigma}$-set, so the set $U$ is Lindelöf. Therefore, by \ref{LindCoz}, the set $U$ is a proximally cozero set.
\end{enumerate}
\end{proof}
In particular, in every Lindelöf proximity space, the proximally Baire sets
coincide with the Baire sets.

\begin{remark}
We may have $O\delta E$ even when $O$ and $E$ are disjoint proximally zero sets.
Furthermore, it is possible that $C\delta C^{c}$ even though $C$ and $C^{c}$ are proximally zero sets.
Give $\mathbb{Z}$ the proximity induced by the one-point compactification. Let
$C$ be the collection of even integers. Then $C^{c}$ is the collection of all odd integers.
Since $C$ and $C^{c}$ are both closed sets we have $C,C^{c}$ be disjoint proximally
zero sets. On the other hand,  $\textrm{cl}_{\mathbb{Z}\cup\{\infty\}}(C)\cap
\textrm{cl}_{\mathbb{Z}\cup\{\infty\}}(C^{c})=\{\infty\}$, so $C\delta C^{c}$.
\end{remark}
\begin{definition}
If $(X,\delta)$ is a proximity space, then let $(X,\delta)_{\aleph_{1}}=(X,\delta_{\mathcal{B}^{*}(X,\delta)})$
be the proximity space equivalent to the $\sigma$-algebra of proximally Baire sets on $X$.
\end{definition}
\begin{theorem}
Let $(X,\mathcal{N})$ be a $\sigma$-algebra, and let $(Y,\delta)$ be
a proximity space. Then a map $f:(X,\delta_{\mathcal{N}})\rightarrow(Y,\delta)$
is a proximity map if and only if $f$ is a proximity map from
$(X,\delta_{\mathcal{N}})$ to $(Y,\delta)_{\aleph_{1}}.$
\end{theorem}
\begin{proof}
$\rightarrow$ Assume that $f:(X,\delta_{\mathcal{N}})\rightarrow(Y,\delta)$
is a proximity map. If $C\subseteq Y$ is a proximally zero set, then
there is a proximity map $g:(Y,\delta)\rightarrow[0,1]$ with $C=g^{-1}(0)$. Therefore
$g\circ f$ is a proximity map as well,  $f^{-1}(C)=f^{-1}(g^{-1}(0))=(g\circ f)^{-1}(0)$
is a proximally zero set, and since $(X,\delta_{\mathcal{N}})$ is a $P_{\aleph_{1}}$-proximity space (by Theorem \ref{SigPProx}),
we obtain, using Corollary \ref{PProxZero} and Theorem \ref{BezDual}, that $f^{-1}(C)\in\mathcal{N}$ for each proximally zero set $C$.
Therefore $f^{-1}(R)\in\mathcal{N}$ for each proximally Baire set $R$.

$\leftarrow$ Now assume that $f:(X,\delta_{\mathcal{N}})\rightarrow(Y,\delta)_{\aleph_{1}}$ is a proximity
map. Then let $C,D\subseteq Y$ be sets with $C\overline{\delta}D$. Then there is a
proximally zero set $Z\subseteq Y$ with $C\subseteq Z,D\subseteq Z^{c}$. Therefore,
by Theorem \ref{ProxMeasTran}, $f^{-1}(Z)\in\mathcal{N}$, and $f^{-1}(C)\subseteq f^{-1}(Z),
f^{-1}(D)\subseteq f^{-1}(Z^{c})$, thus $f^{-1}(C)\overline{\delta_{\mathcal{N}}}f^{-1}(D)$.
\end{proof}
In particular, if $(X,\mathcal{M})$ is a $\sigma$-algebra, and $\delta$ is any proximity
on $\mathbb{R}$ that induces the Euclidean topology, then $f:(X,\mathcal{M})\rightarrow\mathbb{R}$ is measurable
if and only if $f:(X,\delta_{\mathcal{M}})\rightarrow\mathbb{R}$ is a proximity map.

\begin{theorem}
If $(X,\delta)$ is a proximity space, then the topology on $(X,\delta)_{\aleph_{1}}$
is the topology on the $P$-space coreflection of the topology on $X$.
\end{theorem}
\begin{proof}
The proof is left to the reader.
\end{proof}
\begin{lemma}
Let $f:(X,\delta)\rightarrow(Y,\rho)$ be a function. Then $f$ is a proximity map
if and only if whenever $g:(Y,\rho)\rightarrow[0,1]$ is a proximity map, then $g\circ f:(X,\delta)
\rightarrow[0,1]$ is a proximity map.

$\label{492t0eh}$
\end{lemma}
\begin{proof}
$\rightarrow$
If $f$ is a proximity map, then clearly any composition $g\circ f$ must be a proximity map.

$\leftarrow$ Now assume that each composition $g\circ f$ is a proximity map. Assume that $A,B\subseteq Y$
are sets with $A\overline{\rho}B$. Then there is a proximity map $g:(Y,\rho)\rightarrow[0,1]$ with
$A\subseteq g^{-1}(0),B\subseteq g^{-1}(1)$. Therefore
$f^{-1}(A)\subseteq f^{-1}(g^{-1}(0))=(g\circ f)^{-1}(0)$, and
$f^{-1}(B)\subseteq f^{-1}(g^{-1}(1))=(g\circ f)^{-1}(1)$, so $f^{-1}(A)\overline{\delta}f^{-1}(B)$ since
$g\circ f$ is a proximity map.
\end{proof}
We shall write $\chi_{A}$ for the characteristic function on $A$. In other words,
$\chi_{A}(A)=1$ and $\chi_{A}(A^{c})=0$.

\begin{theorem}
Let $(X,\delta)$ be a proximity space. Then the following are equivalent.

\begin{enumerate}
\item $(X,\delta)$ is a $P_{\aleph_{1}}$-proximity space.

\item If $f_{n}:(X,\delta)\rightarrow[0,1]$
is a proximity map for each $n\in\mathbb{N}$, and $f_{n}\rightarrow f$ pointwise
(here we do not assume $f$ is continuous), then
$f:(X,\delta)\rightarrow[0,1]$ is also a proximity map. 

\item For each proximity space $(Y,\rho)$, if $f_{n}:(X,\delta)\rightarrow(Y,\rho)$
is a proximity map for each $n\in\mathbb{N}$, and $f_{n}\rightarrow f$ pointwise, then
$f:(X,\delta)\rightarrow(Y,\rho)$ is also a proximity map. 

\end{enumerate}

$\label{reaefioeai}$
\end{theorem}
\begin{proof}
$1\rightarrow 2$
The map $f$ is a proximity map if and only if $f$ is a measurable function on the $\sigma$-algebra
$(X,\mathcal{M}_{\delta})$. The implication follows since measurable
functions are closed under pointwise convergence.

$2\rightarrow 3$
Assume that $(Y,\rho)$ is any proximity space, and assume that $f_{n}\rightarrow f$
pointwise and each $f_{n}$ is a proximity map from $(X,\delta)$ to $(Y,\rho)$. Then let $g:(Y,\rho)\rightarrow[0,1]$
be a proximity map. Then $g\circ f_{n}\rightarrow g\circ f$ pointwise, and since
each $g\circ f_{n}:(X,\mathcal{M})\rightarrow[0,1]$ is a proximity map, we have
$g\circ f$ also be a proximity map. Therefore $f$ is a proximity map by
Lemma $\ref{492t0eh}$.

$3\rightarrow 2$ This is trivial.

$2\rightarrow 1$
Assume $Z\subseteq X$ is a proximally zero set. Then there is a
proximity map $f:(X,\delta)\rightarrow[0,1]$ with $Z=f^{-1}(1)$. On the other hand, 
we have $f^{n}\rightarrow\chi_{Z}$ pointwise where $\chi_{Z}$ denotes the characteristic function, 
so $\chi_{Z}$ is a proximity map.
Therefore $Z\overline{\delta}Z^{c}$, so $Z\in\mathcal{M}_{\delta}$. Thus, using Corollary \ref{PProxZero},
we conclude that $X$ is a $P_{\aleph_{1}}$-proximity space.
\end{proof}

\begin{theorem}
Let $X$ be a Tychonoff space. Then $X$ is a $P$-space if and only if whenever
$f_{n}:X\rightarrow[0,1]$ is continuous for all $n$, and $f_{n}\rightarrow f$ pointwise,
then $f$ is continuous.
$\label{wqriu31932}$
\end{theorem}
\begin{proof}
$\rightarrow$ Assume $X$ is a $P$-space. Then for each continuous
$f_{n}:X\rightarrow[0,1]$ and $x\in X$ there is an open neighborhood $U_{n}$ of $x$ where
$f_{n}(U_{n})=f_{n}(x)$. Therefore if $U=\bigcap_{n}U_{n}$, then $U$ is an open
neighborhood of $x$. Furthermore, if $f_{n}\rightarrow f$ pointwise, then
$f(U)=f(x)$. Therefore $f$ is continuous at each point $x\in X$.

$\leftarrow$ Assume $Z\subseteq X$ is a zero set. Then let
$f:X\rightarrow[0,1]$ be a continuous function where $f^{-1}(1)=Z$. Then 
$f^{n}\rightarrow\chi_{Z}$ pointwise. Therefore since $\chi_{Z}$ is continuous,
we have $Z$ be open. Therefore $X$ is a $P$-space.
\end{proof}
\begin{corollary}
If $X$ is Tychonoff and $\delta$ is the proximity induced by the
Stone-Cech compactification of $X$, then $X$ is a $P$-space if and only if $(X,\delta)$ is a
$P_{\aleph_{1}}$-proximity space.
\end{corollary}
\begin{proof}
The proximity maps from $(X,\delta)$ to $[0,1]$ are precisely the
continuous functions from $X$ to $[0,1]$. Therefore by theorems $\ref{reaefioeai}$ and
$\ref{wqriu31932}$ we have $X$ be a $P$-space if and only if $(X,\delta)$ is a $P_{\aleph_{1}}$-proximity space.
\end{proof}

\section{\bf Conclusions and applications}
We conclude this paper by demonstrating that it is sometimes better to consider
$\sigma$-algebras as $P_{\aleph_{1}}$-proximities since proximity spaces are often easier to work with than
$\sigma$-algebras. 

If $(X,\mathcal{M})$ is a $\sigma$-algebra, then let $L^{\infty}(X,\mathcal{M})$
denote the collection of all bounded measurable functions from $(X,\mathcal{M})$ to $\mathbb{C}$. Clearly
$L^{\infty}(X,\mathcal{M})$ is a Banach-algebra and even a $C^{*}$-algebra.
If $Y$ is a compact space, then let $C(Y)$ be the Banach-algebra which consists of all continuous
functions from $Y$ to $\mathbb{C}$. Let $\mathcal{C}$ be the Smirnov compactification of 
$(X,\delta_{\mathcal{M}})$. Then $L^{\infty}(X,\mathcal{M})$ is isomorphic as a Banach-algebra to $C(\mathcal{C})$.
One can easily show that $\mathcal{C}$ is the collection of all ultrafilters on the Boolean
algebra $\mathcal{M}$. The maximal ideal space of $L^{\infty}(X,\mathcal{M})$ is therefore homeomorphic
to the collection of all ultrafilters on $\mathcal{M}$. Furthermore, from these facts one can easily show that if
$(X,\mathcal{M},\mu)$ is a measure space, then the maximal ideal space of $L^{\infty}(\mu)$ is
homeomorphic to the collection of all ultrafilters on the quotient Boolean algebra
$\mathcal{M}/\{R\in\mathcal{M}|\mu(R)=0\}$.

\bibliographystyle{plain}

\end{document}